\documentclass{article}
\usepackage{amsmath,amsfonts,amsthm,amssymb,amscd,color,xcolor,mathrsfs,eufrak, mathtools,mathrsfs,enumitem, upgreek}
\usepackage{microtype}
\usepackage{hyperref}

 \usepackage{comment}

\colorlet{darkblue}{blue!50!black} 
\colorlet{darkmagenta}{magenta!80!black}

\hypersetup{ 
    colorlinks,%
    citecolor=darkblue,%
    filecolor=red,%
    linkcolor=darkblue,%
    urlcolor=blue,%
    pdfnewwindow=true,%
    pdfstartview={FitH}
}

\newcommand{\p}{\partial}
\newcommand{\e}{\varepsilon}

 \DeclareMathAlphabet\mathbfcal{OMS}{cmsy}{b}{n}

\newcommand{\R}{{\mathbb R}}
\newcommand{\Z}{{\mathbb Z}}

\newcommand{\T}{{\mathbb T}}

\newcommand{\La}{\Lambda}

\newcommand{\ty}{\infty}

\newcommand{\de}{\delta}

\newcommand{\aA}{{\cal A}}

\newcommand{\CC}{{\cal C}}

\newcommand{\FF}{{\cal F}}
\newcommand{\GG}{{\cal G}}
\newcommand{\HH}{{\cal H}}
\newcommand{\II}{{\cal I}}

\newcommand{\RR}{{\cal R}}

\newcommand{\TT}{{\cal T}}

\newcommand{\XX}{{\cal X}}

\newcommand{\lag}{\langle}
\newcommand{\rag}{\rangle}

\newcommand{\dd}{{ \,\textup d}}

\newcommand{\lspan}{\mathop{\rm span}\nolimits}

\theoremstyle{plain}
\newtheorem*{mtheorem}{Main Theorem}

\newtheorem{theorem}{Theorem}[section]
\newtheorem{lemma}[theorem]{Lemma}
\newtheorem{proposition}[theorem]{Proposition}

\theoremstyle{definition}
\newtheorem{definition}[theorem]{Definition}

\theoremstyle{remark}

\numberwithin{equation}{section}

\makeatletter
\let\@fnsymbol\@arabic 
\makeatother

\begin{document} 
\author{Vahagn~Nersesyan\footnote{Universit\'e Paris-Saclay, UVSQ, CNRS, Laboratoire de Math\'ematiques de Versailles, 78000, Versailles, France  \& Centre de Recherches Math\'emati\-ques,~CNRS~UMI 3457, Universit\'e de Montr\'eal, Montr\'eal,  QC, H3C 3J7, Canada;    e-mail: \href{mailto:Vahagn.Nersesyan@math.uvsq.fr}{Vahagn.Nersesyan@math.uvsq.fr}}}
 
\title{Approximate controllability of nonlinear parabolic PDEs in arbitrary space dimension}
\date{\today}
\maketitle

\smallskip

\begin{abstract}

In this paper, we consider a   parabolic PDE   on a torus of arbitrary  dimension.~The nonlinear term   is a   smooth function  of polynomial growth of any       degree.~In this general setting,  the      Cauchy problem is not necessarily well posed.~We show that the equation in question      is approximately controllable by only  a finite number of 
     Fourier modes.~This result is proved by using some ideas~from the   
        geometric control    theory introduced by Agrachev and Sarychev.


\end{abstract}

\tableofcontents

\setcounter{section}{-1}

\section{Introduction}
\label{s0} 

We consider the following parabolic PDE on  the $d$-dimensional torus $\T^d$:
\begin{equation} \label{0.1}
	\p_t u-\nu\Delta u+f(u)=h(t,x)+\eta(t,x),   \quad (t,x)\in (0,T)\times \T^d, \,d\ge1,
\end{equation}
where   $\nu$  is a positive number,     $h:[0,T]\times \T^d\to \R$ is a given  smooth function,   and $f:\R\to \R$ is a   nonlinear term. The latter is assumed to be  of the form
\begin{equation}\label{0.2}
  f(y)= P_p(y)+g(y), 
\end{equation} where $P_p$ is a polynomial of degree   $p\ge2$ and    $g:\R\to \R$ is a bounded smooth function with bounded derivatives.     Let us emphasise  that the   parameters~$d, p,$ and~$T$  are arbitrary, so that Eq.~\eqref{0.1} supplemented with the initial~condition
\begin{equation} \label{0.4}
	u(0)=u_0
\end{equation}is not necessarily    well~posed on the time interval $J_T:=[0,T]$.~For example,  see Section~17 in the book~\cite{QS-2007} and the references therein  for   constructions of finite time blow-up solutions for   problem~\eqref{0.1},~\eqref{0.4}.   We shall take initial condition in the Sobolev space $H^s(\T^d)$ with $s>d/2$, in order to have   locally well-posed Cauchy problem, i.e., local existence, uniqueness, and continuous dependence on the initial condition and  the source term (see   Proposition~\ref{P:1.1}).

The function~$\eta$   plays the role of the  control and   is assumed to be {\it degenerate} in the Fourier space. More precisely, $\eta$   takes values in a  finite-dimensional  space   defined~by 
\begin{equation}\label{0.5}
 \HH(\II)=\lspan\left\{\sin\lag x,k\rag, \cos\lag x,k\rag: k\in \II\right\},
\end{equation}
 where   $\II\subset \Z^d$  is a    finite   symmetric set (i.e., $\II=-\II$)        containing the origin and~$\lag \cdot,\cdot \rag$ is the   Euclidean scalar product in $\R^d$.~Recall that $\II$ is called a    generator  if~any element of  $\Z^d$ is a   linear combination of   elements of $\II$ with integer coefficients.

 To formulate the main result of this paper, let us fix any   $s> d/2$,~$T>0$, and $h\in L^2(J_T,H^{s-1}(\T^d))$. 
  We shall say that Eq.~\eqref{0.1} is   {\it approximately~controllable}~by $\HH(\II)$-valued~control if  for any   initial point $u_0\in H^s(\T^d)$, any target $u_1\in H^s(\T^d)$,   and any   number  $\e>0$, there is a control~$\eta\in L^2(J_T, \HH(\II))$ and a unique  solution~$u$  of problem \eqref{0.1}, \eqref{0.4} defined on the interval $J_T$     such~that
 $$
 \|u(T)-u_1\|_{H^s}<\e.
 $$   
  \begin{mtheorem}
   Assume that one of the following two conditions hold:
a)~$g=0$; or
b)	$p>s     > d/2$ and $g$ is arbitrary function as above. 
If~$\II$ is a generator, then Eq.~\eqref{0.1} is approximately controllable by~$\HH(\II)$-valued control.	
 \end{mtheorem}   The interpretation of the  part  b)   is that a sufficiently strong polynomial component is needed in the nonlinearity in order to brake the influence of the perturbation~$g$. 
 See Section~\ref{S:2} for   more general results.~In particular, in the case when~$f$   is a polynomial (i.e., $g=0$) and $h=0$, the condition that the set $\II$ is a generator is also necessary for   approximate  controllability    (see Theorem~\ref{T:2.5}).   Note that 
the condition on~$\II$       is completely independent of the choice of the functions    $f$ and $h$ and the     parameters~$\nu, s, p,$ and~$T$.

  The proof of the Main Theorem uses some  arguments from the works of   Agrachev and Sarychev~\cite{AS-2005, AS-2006, AS-2008}, who studied   the approximate controllability of the 2D Navier--Stokes (NS)  and Euler systems   by   finite-dimensional forces. Their approach has been   extended to different equations by many authors.    Shirikyan~\cite{shirikyan-cmp2006, shirikyan-aihp2007} established     the approximate  controllability of the   3D NS     system  on the   torus. He also   considered the Burgers equation    on the real line in~\cite{Shi-2013}   and      on a bounded  interval with   Dirichlet boundary conditions   in~\cite{Shir-2018}.   Rodrigues~\cite{SSRodrig-06} proved  approximate controllability    of the 2D NS system  on a  rectangle with   Lions boundary conditions, and with  Phan~\cite{RD-2018} they generalised that result to the 3D case.  
In the papers~\cite{Hayk-2010,   Hayk-2011}, Nersisyan considered 3D Euler system for  
incompressible  and compressible  fluids, and   Sarychev   \cite{Sar-2012} considered the  2D defocusing cubic Schr\"odinger equation.   The controllability of the Lagrangian trajectories of the 3D NS system is considered in~\cite{Ners-2015}   by the author.    

We   use a   technique of applying large controls on small time intervals inspired by the works of Jurdjevic and  Kupka (see the paper~\cite{JK-1985} and Chapter~5 in the book~\cite{MR1425878}), who  considered       
finite-dimensional control systems.~Infinite-dimensional   generalisations of this approach appear in the above-mentioned papers of  Agrachev and Sarychev (e.g., see Section~6.2 in~\cite{AS-2006}) and in the paper~\cite{GHM-2018}   of   Glatt-Holtz, Herzog, and  Mattingly. In the latter, the authors prove, in particular, 
      approximate controllability of a 1D parabolic PDE    with polynomial nonlinearity of odd degree.

   Without going   into  the technical details, let us describe    some ideas of the proof of the Main Theorem.  Together with Eq.~\eqref{0.1},
     we consider an   equation of the~form 
\begin{equation} \label{0.6}
  	\p_t u-\nu\Delta (u+\zeta)+f(u+\zeta)=h(t,x)+\eta(t,x)
\end{equation}
with   $\zeta$ and $\eta$  taking values in $\HH(\II)$. It   turns out   that Eq.~\eqref{0.1} is approximately controllable with control $\eta$ if and only if so is Eq.~\eqref{0.6} with two controls~$\zeta$ and~$\eta$.
   The   solution $u$  of problem~\eqref{0.6},~\eqref{0.4},  whenever   exists, is denoted by 
  \begin{equation}\label{solu(t)}
  	  \RR_t(u_0,\zeta,h+\eta):=u(t).
  \end{equation}
The first step is the   following asymptotic property that holds for any smooth   functions~$\zeta$  and $\eta$ not depending on time:
\begin{equation}\label{0.7}
 	 \RR_{\de}(u_0,\de^{-1/p}\zeta,h+\de^{-1}\eta)\to  u_0+\eta-c\zeta^p \quad\text{in $H^s(\T^d)$ as $\de\to 0^+$}, 
\end{equation}
where $c$ is the leading coefficient of the polynomial $P_p$ in \eqref{0.2}.
   This allows to steer the trajectory of~\eqref{0.6},~\eqref{0.4} in small time    close to   any target $u_1$ belonging to the affine space~$u_0+\HH_1(\II)$, where $\HH_1(\II)$ is the   largest 
vector space   whose elements can be   written in the form
\begin{equation}\label{Reqn}
\eta-\sum_{m=1}^n\zeta_m^p
\end{equation}
for  some integer   $n\ge1$ and   vectors $\eta,\zeta_1, \ldots,\zeta_n\in \HH(\II)$ (see Section~\ref{S:2} for the precise definition of $\HH_1(\II)$). Then iterating this argument,   we show that starting from $u_0$ we can also attain approximately   any point in $u_0+\HH_2(\II)$, where~the space $\HH_2(\II)$  is defined     by \eqref{Reqn}, but now with vectors 
  $\eta,\zeta_1, \ldots,\zeta_n\in \HH_1(\II)$. In~this way,  we construct  a non-decreasing sequence of subspaces $\{\HH_j(\II)\}$ such that the points in $u_0+\HH_j(\II)$ are attainable from $u_0$. From the fact that $\II$ is a generator we deduce that  the union  $\cup_{j=1} \HH_j(\II)$ is dense in~$H^s(\T^d)$ (i.e., $\HH(\II)$ is  {\it saturating} in the language of the geometric control theory). This allows to control approximately Eq.~\eqref{0.6} to any point in $H^s(\T^d)$ in small time. The   controllability in any time $T$ is derived by
steering the system close to the target~$u_1$ in small time, then by    keeping the trajectory   close to $u_1$   for a sufficiently long period of time, by   applying an appropriate control. 
 
    The main contribution of this paper  is the generality of the assumptions on the nonlinear term. We give a simple   condition on the set  of the   Fourier modes~$\II$ that ensures the controllability of the system.~Surprisingly,  the condition is independent of the nonlinear term and is also    necessary in the polynomial case. The perturbation term $g$ in~\eqref{0.2} brings   new   difficulties that do not appear in the previously considered situations. Indeed,     the nonlinear  term~$f$ now      may     vanish on   some  ball, so it will not be able to  couple there the   Fourier modes in $\II$ to the others. This situation seems to be out of the reach with previous~methods.  The main ingredient in the above scheme of the proof  is the limit \eqref{0.7}. According to that limit,    when applying     large controls~$\delta^{-1/p}\zeta$ on small  time intervals $[0,\delta]$,  we only see the highest order term~$cy^p$ of the nonlinearity~$f(y)$ when we pass  to the limit~$\delta\to 0^+$. Thus the contribution of  the perturbation   $g$ vanishes in the limit, provided that the condition b)~is~satisfied.  
  
  The proof of approximate controllability we give   is   short and conceptually simple; it is quite    general  and can be adapted to more degenerate problems. In~\cite{BGN20},  this approach is further developed to consider the controllability of the system of 3D primitive equations of geophysical fluid dynamics  with control acting directly only on the temperature equation.



  Finally, let us mention that this paper is partially motivated  by applications to the ergodicity of  randomly forced PDEs. Indeed, the control theory is known to be a useful  tool in the study of stochastic systems   with   highly degenerate noise.   We refer the reader to the paper~\cite{KNS-2019} for more details  and references   on this subject     and for a concrete application of our Main Theorem   in the study of a stochastic version of  Eq.~\eqref{0.1}, i.e., when  the control~$\eta$ is replaced by a random process.

  This paper is organised as follows. In Section~\ref{S:1}, we establish  a perturbative result on the  existence and stability of solutions of  problem~\eqref{0.1},~\eqref{0.4}. The proof of the Main Theorem is given in  Section~\ref{S:2}.~In Section~\ref{S:3}, we prove limit~\eqref{0.7}, and in
    Section~\ref{S:4}, we construct   examples of saturating spaces.

    \subsubsection*{Acknowledgement}
    
The author thanks the referees for their valuable comments and suggestions.     The author thanks Armen Shirikyan for helpful    discussions.~This research was supported by the {\it Agence Nationale de la Recherche} through  the grant NONSTOPS   ANR-17-CE40-0006-02 and by the CNRS through  the PICS grant {\it Fluctuation theorems in stochastic systems}.

 \subsection*{Notation} 
 
 In this paper, we use the following notation.

\smallskip
\noindent
$\Z^d$ is   the integer lattice in~$\R^d$.

\smallskip
\noindent
 $\T^d$ is the standard $d$-dimensional torus~$\R^d/2\pi\Z^d$.

\smallskip
\noindent
$H^s:=H^s(\T^d)$ is the Sobolev space of order $s$ endowed with the usual norm~$\|\cdot\|_s$.

\medskip

Let  $X$ be a Banach  space endowed with the
norm $\|\cdot\|_X$ and let $J_T:=[0, T ]$.

\smallskip
\noindent
$B_X(a,r)$ is  the closed ball of radius $r>0$ centred at $a\in X$. We write $B_X(r)$, when $a=0$.

\smallskip
\noindent
$L^q(J_T,X),$  $1\leq q<\infty$   is      the
space of measurable functions $u: J_T \rightarrow X$ endowed with the norm
\begin{equation}
\|u\|_{L^q(J_T,X)}:=\bigg(\int_{0}^T \|u(t)\|_X^q\dd t
\bigg)^{1/q}<\infty.\nonumber
\end{equation}

\noindent
$L^q_{loc}(\R_+,X)$ is the space of measurable functions $u: \R_+ \rightarrow X$  whose restriction to the interval $J_T$ belongs to  $L^q(J_T,X)$ for any  $T>0$.

\smallskip
\noindent
$C(J_T,X)$ is the space of continuous functions $u:J_T\to X$ with the norm 
$$
\|u\|_{C(J_T,X)}:=\max_{t\in J_T} \|u(t)\|_X.
$$

\noindent
  $x\wedge y$  denotes the     minimum of real numbers $x$ and $y$.
 
 \smallskip
 \noindent
 $C,\, C_1,\, \ldots$   denote some  unessential positive constants.

\section{Local well-posedness and  stability}\label{S:1}

  In this section, we study  the   local  existence and stability  of solutions for the following  generalisation of Eq.~\eqref{0.1}:      
\begin{equation}
	\p_t u-\nu\Delta (u+\zeta)+f(u+\zeta)=\varphi,
\label{1.1}
\end{equation}   where $f$ is of the form \eqref{0.2} with any polynomial $P_p$ of degree $p\ge2$ and any bounded smooth function $g:\R\to \R $         with bounded derivatives.   In this paper, by smooth we always mean $C^\ty$-smooth. 
 For any~$T>0$ and any integer~$s>d/2$, we define the space 
$$
\XX_{T,s}:= C(J_T,H^{s})\cap L^2(J_T,H^{s+1})
$$ and endow it with the norm 
$$
\|u\|_{\XX_{T,s}}:=\|u\|_{C(J_T,H^{s})}+ \|u\|_{ L^2(J_T,H^{s+1})}.
$$  
  \begin{proposition}\label{P:1.1}
Let     $\hat u_0\in H^{s}$,   $\hat \zeta \in C
(\R_+,H^{s+1})$, and $\hat \varphi\in L^2_{loc}(\R_+,H^{s-1})$. There is a maximal time     $T_*:=T_*(\hat u_0, \hat \zeta, \hat \varphi)>0$ and a unique solution $\hat u$ of problem~\eqref{1.1},~\eqref{0.4} with $(u_0, \zeta, \varphi)=(\hat u_0, \hat \zeta, \hat \varphi)$     whose restriction to $J_T$  belongs  to~$\XX_{T,s}$ for any $T<T_*$. If~$T_*<\ty$, then   
\begin{equation}\label{1.2}
\|\hat u(t)\|_s\to +\infty \quad \text{as $t\to T_*^-$}.
\end{equation} 
Furthermore, for any $T<T_*$,  there are positive 
  constants
 $\delta=\delta(T,\Lambda)$  and $C=C(T,\Lambda)$, where   
 \begin{equation}\label{1.3}
\Lambda:= \|\hat \zeta\|_{C(J_T,H^{s+1})}+\|\hat \varphi\|_{L^2(J_T,H^{s-1})}+ \|\hat u\|_{\XX_{T,s}},
\end{equation}
such that the following properties hold.

\begin{enumerate}
\item[(i)] For any $u_0\in H^{s},$ $\zeta\in C(J_T,H^{s+1})$, and $\varphi\in L^2(J_T,H^{s-1})$  satisfying  
\begin{equation}\label{1.4}
\|u_0-\hat u_0\|_{s}+ \|\zeta-\hat\zeta\|_{C(J_T,H^{s+1})} + \|\varphi-\hat \varphi\|_{L^2(J_T,H^{s-1})}<
\delta,
\end{equation}
  problem \eqref{1.1}, \eqref{0.4} has a unique
solution $ u\in \XX_{T,s}.$
\item[(ii)] 
As in \eqref{solu(t)}, let $\RR$ be the resolving operator for \eqref{1.1}, i.e., the mapping  taking a triple $(u_0,\zeta,\varphi )$ satisfying \eqref{1.4} to the solution~$u$.
 Then     
 \begin{align}\label{1.5}
\|\RR(u_0,\zeta,\varphi) -\RR(\hat u_0,\hat\zeta,\hat \varphi)\|_{\XX_{T,s}}&\le 
C \big( \|u_0-\hat u_0\|_s+
\|\zeta-\hat\zeta\|_{C(J_T,H^{s+1})}\nonumber\\&\quad +\|\varphi-\hat \varphi\|_{L^2(J_T,H^{s-1})}\big).
 \end{align}
 \end{enumerate}
\end{proposition}
\begin{proof}  
  Local existence of a solution $\hat u$  and   \eqref{1.2} are proved  using a fixed point approach based on estimates that we will use   in the proof of (i) and~(ii) below. As the argument is quite standard, we skip the details of the proof of that  part.  

{\it Step~1.~Proof of   (i).}
Let $T<T_*(\hat u_0, \hat \zeta, \hat \varphi)$,  and  let $( u_0,  \zeta,  \varphi)$ be as in (i).
  We extend $\zeta$ and $\varphi$ by zero outside the interval $J_T$ and denote by $u$ the corresponding solution. The latter exists up to some maximal time $T_*( u_0,  \zeta,  \varphi)>0$ by the first part of the proposition.      Let us show that $T<T_*( u_0,  \zeta,   \varphi)$, provided that~$\delta>0$ in~\eqref{1.4} is sufficiently small. Indeed, the function $w=u-\hat u$   is  a solution of   the  problem
\begin{align}
&\p_t w -\nu \Delta (w+\xi) +f(w+\xi+\hat u+\hat\zeta)-f(\hat u+\hat\zeta) =\psi,\label{1.6}\\
&w(0,x)=w_0(x) \label{1.7}
\end{align}
with  $w_0=u_0-\hat u_0$, $\xi=\zeta-\hat \zeta$, and $\psi=\varphi-\hat \varphi$.
For any $\TT>0$ and $F\in L^2(J_\TT,H^{s-1})$, we have that 
$$
\Psi   := e^{\nu t\Delta}w_0+\int_0^te^{\nu (t-\tau)\Delta}  F  \dd \tau  
$$belongs to $\XX_{\TT,s} $ and 
\begin{equation}\label{E:Psi}
\|\Psi \|_{\XX_{\TT,s}}\le C_1 \left( \|w_0\|_s+ \|F \|_{L^2(J_\TT,H^{s-1})}   \right),
\end{equation}  where $C_1$ does not depend on $\TT$. Let us denote
$$
\TT=\sup\{t<  T_*(\hat u_0, \hat \zeta, \hat \varphi) \wedge T_*( u_0,  \zeta,  \varphi): \|w(t)\|_s<1\}.
$$We will show that $\TT>T$, provided that $\delta$ in \eqref{1.4} is sufficiently small. To this end, we apply \eqref{E:Psi} with
$$
F :=\nu\Delta\xi -f(w+\xi+\hat u+\hat\zeta)+f(\hat u+\hat\zeta)+ \psi,
$$ and note that   $w=\Psi$ for this choice of $F$. Then 
\begin{align}\label{1.8}
\|w \|_{\XX_{\TT,s}}&\le C_1 \big( \|w_0\|_s+ \nu \|\xi \|_{L^2(J_\TT,H^{s+1})}+ \|\psi \|_{L^2(J_\TT,H^{s-1})}\nonumber  \\&\quad + \|  f(w+\xi+\hat u+\hat\zeta)-f(\hat u+\hat\zeta) \|_{L^2(J_\TT,H^{s-1})}   \big).
\end{align}
To  estimate  the   term with    $f$,  we start with the polynomial part: 
\begin{align}\label{1.9}
 \|P_p(w+\xi+\hat u&+\hat\zeta)-P_p(\hat u+\hat\zeta)\|_s \nonumber\\&\le C_2\, \|w+\xi\|_s \Big(\|w\|_s+\|\xi\|_s+\|\hat u\|_s+\|\hat\zeta\|_s+1\Big)^{p-1},
\end{align} where we   used the fact that $H^s$ is an algebra for~$s>d/2$ and    
\begin{equation}\label{1.10}
\|ab\|_s\le C_3\, \|a\|_s \|b\|_s, \quad a,b\in H^s.
\end{equation}Then we write 
$$
g(w+\xi+\hat u+\hat\zeta)-g(\hat u+\hat\zeta)=(w+\xi) \int_0^1g'(\tau(w+\xi)+\hat u+\hat\zeta)\dd \tau,  
$$and apply inequality    \eqref{1.10}:
\begin{equation}\label{E:gna}
\|g(w+\xi+\hat u+\hat\zeta)-g(\hat u+\hat\zeta)\|_s\le 	\| w+\xi\|_s \int_0^1 \|g'(\tau(w+\xi)+\hat u+\hat\zeta)\|_s\dd \tau.
\end{equation}
Now we use the   inequality
\begin{equation}\label{EEgna}
\|g'(a)\|_s\le C_4 \, \|a\|_s^s, \quad a\in H^s,
\end{equation}which is obtained by applying the Sobolev inclusion $H^s\subset L^\ty $ and the Gagliardo--Nirenberg inequality
$$
\|D^k a\|_{L^{2s/k}}\le C_5\, \|D^s a\|^{k/s}_{L^2}\|   a\|_{\ty}^{1-k/s}
$$
to the   terms that arise in estimating the $L^2$ norms of the derivatives
of $g'(u)$ (see Section~2 in \cite{A76}), where $D^k a$ is the derivative of order $k$ of the function $a$. Combining \eqref{E:gna} and \eqref{EEgna}, we get
$$
\|g(w+\xi+\hat u+\hat\zeta)-g(\hat u+\hat\zeta)\|_s\le C_6 \,	\| w+\xi\|_s  \left( \|w\|_s+\|\xi\|_s+\|\hat u\|_s+\|\hat\zeta\|_s+1 \right)^s.
$$
This inequality, together with \eqref{1.4}, \eqref{1.9}, and the Young inequality, implies that 
\begin{equation}\label{RRRE}
\|f(w+\xi+\hat u+\hat\zeta)-f(\hat u+\hat\zeta)\|_s\le C_7  \left(\delta +	\|w\|^m_s \right), \quad t <  \TT,
\end{equation}where  $m= p\wedge (s+1)$ and  $C_7=C_7(\La)$. Going back to \eqref{1.8}, we obtain
\begin{equation}\label{TTR}
\|w(t)\|_{\XX_{\TT,s}}^2\le C_8 \left(\delta + \int_0^t \|w(\tau)\|_s^{2m}\dd \tau  \right), \quad t< \TT\wedge T.
\end{equation} Let us denote
$$
\Phi(t):=  \delta + \int_0^t \|w(\tau)\|_s^{2m}\dd \tau . 
$$Inequality \eqref{TTR} implies that 
$$
(\dot \Phi)^{1/m}\le C_8\,\Phi, 
$$ which is equivalent to
 $$
\frac{\dot \Phi}{\Phi^{m}}\le C_8^{m}.
$$Integrating the latter, we get
$$
\Phi(t)\le \begin{cases} \delta \exp(C^m_8 t), & \text{if } m=1, \\ \delta (1-(m-1) C_8^m\delta^{m-1} t)^{-1/(m-1)}, & \text{if }m\ge2, \end{cases} \quad \quad t< \TT\wedge T.
$$Choosing $\delta$ sufficiently small, we see that 
\begin{equation}\label{EElam}
\Phi(t)< C_9 \delta<1 , \quad   t< \TT\wedge T,
\end{equation}which implies that $\TT>T$ and  proves (i).

\medskip
{\it Step~2.~Proof of   (ii).} Using inequalities~\eqref{1.8}, \eqref{RRRE}, \eqref{TTR}, and \eqref{EElam}, we~get 
\begin{align*} 
\|w\|_{\XX_{t,s}}^2&	\le C_{10}\left(\|w_0\|_s^2+ \|\xi\|_{L^2(J_T,H^{s+1})}^2+\|\psi\|^2_{L^2(J_T,H^{s-1})} +\int_0^t \|w\|_{\XX_{\tau,s}}^2 \!\dd\tau\right)
\end{align*}for any $t\in J_T$.
Applying the Gronwall inequality,  we obtain~\eqref{1.5} and   complete  the proof of the proposition. 
    \end{proof}

\section{Main result}\label{S:2}

Let us  take any $s>d/2$, $T>0$, 
  $h\in L^2  (J_T,H^{s-1})$, and $u_0\in
H^s$, and consider problem~\eqref{0.1},~\eqref{0.4}. The function  $\eta$ will be the  control   taking  values in a finite-dimensional subspace $\HH$ of $H^{s+2}$ that will be specified below.   
 Let us set 
 $$
   \Theta(u_0,h,T):=\{\eta\in
L^2(J_T,H^{s-1}): \text{s.t.~\eqref{0.1},~\eqref{0.4} has
a solution in~$\XX_{T,s}$}\},
 $$ and recall that 
 $\RR(\cdot,\cdot,\cdot)$  is the  resolving operator 
 for~\eqref{1.1},~\eqref{0.4} and  $\RR_t(\cdot,\cdot,\cdot)$ is  its 
   evaluation    at time $t$. 
   \begin{definition}\label{D:2.1}  We shall say that 
  Eq.~\eqref{0.1} is     approximately controllable
by   $\HH$-valued control  if  for any $\e>0$   and   any      $u_0,u_1\in H^s$     there is  a control $\eta
\in \Theta(u_0,h,T)\cap L^2(J_T,\HH) $ such that
$$
\|\RR_T(u_0,0,h+\eta) - u_1 \|_s<\e.
$$   
\end{definition} 

  To simplify the presentation, we   assume that the   leading coefficient $c$ of the polynomial $P_p$  in \eqref{0.2} equals~to one. Indeed, the general case can be reduced to this one by a time scaling~$\tau=c^{-1}t$  and noting that the below results do not change if we multiply~$\nu, h$ and~$g$ by a constant.

 Recall that  $p\ge2$ is the integer in~\eqref{0.2}.~For any finite-dimensional subspace~$\HH\subset  H^{s+2}$,   let $\CC(\HH)$ be
   the (nonconvex) cone defined by 
   $$
   \left\{ \eta - \sum_{m=1}^{n}\zeta_m^p :  \ n \geq 1, \ \eta , \zeta_1, \ldots , \zeta_n  \in {\mathcal H}      \right\}.
   $$We denote by~$\FF(\HH)$   the largest 
vector space contained in  $H^{s+2}\cap \overline {\CC(\HH),}^{H^s}$ where~$\overline {\CC(\HH)}^{H^s}$ is the closure of~$\CC(\HH)$ in~$H^s$.  It is easy to check that~$\FF(\HH)$ is well defined and     finite-dimensional. Iterating this, we construct a
         non-decreasing sequence of finite-dimensional subspaces:    
\begin{equation}\label{2.1}
\HH_0=\HH,\quad \HH_j=\FF(\HH_{j-1}),  \quad j \geq
1, 
\end{equation} and denote $ \HH_\infty=\bigcup_{j=1}^\infty \HH_j.$
\begin{definition}\label{D:2.2}
 We   say that $\HH$ is  saturating  if $\HH_\ty$  is dense in $ H^s$.
\end{definition}
  Our definition of  saturating subspace   is close  to the  definitions  used in   the papers~\cite{AS-2006, shirikyan-cmp2006}. The difference is that here the subspaces~$\HH_j$  are defined in terms of the  approximations of the elements of the form~\eqref{eeeZ}  and not by the elements themselves. This point is   important for the examples of saturating subspaces   given
  in   Section~\ref{S:4}.

 The following is a more general version of the  Main Theorem stated in the Introduction. 
\begin{theorem}\label{T.2.3}   Assume that one of the following two conditions hold:
a)~$g=0$; or
b)	$p>s    > d/2 $ and $g:\R\to\R$ is an arbitrary  bounded smooth function with bounded derivatives. 
If $\HH 
$  is saturating, then Eq.~\eqref{0.1}  is approximately controllable    by     $\HH$-valued control.
\end{theorem}
We derive this theorem   from the following proposition proved in Section~\ref{S:3}. 
Let us~denote 
\begin{align*}
   \hat{\Theta}(u_0,h,T):=\{&  (\eta, \zeta)\in L^2(J_T,H^{s-1})\times
C(J_T,H^{s+1}): \text{s.t.~\eqref{1.1},~\eqref{0.4}}\\&   \text{has
a solution in~$\XX_{T,s}$ with    $\varphi=h+\eta$}\}.
\end{align*}
 \begin{proposition}\label{P:2.4} Under the conditions of Theorem~\ref{T.2.3},
 for any  $u_0,\eta\in H^{s+1}$, $\zeta \in  H^{s+2}$, and   $h\in L^2  (J_T,H^{s-1})$, there is  $\de_0>0$ such that $(\de^{-1/p}\eta,\de^{-1}\zeta)\in \hat{\Theta}(u_0,h,\de)$ for any $\de\in (0,\de_0)$, and     the following  limit  holds   for the corresponding solution at time $t=\delta$: 
$$
 	 \RR_{\de}(u_0,\de^{-1/p}\zeta,h+\de^{-1}\eta)\to  u_0+\eta-\zeta^p \quad\text{in $H^s$ as $\de\to 0^+$}.
$$ 
\end{proposition}
The proof of  this proposition will be given in Section~\ref{S:3}.  
\begin{proof}  [Proof of Theorem \ref{T.2.3}]
  As discussed in the Introduction, the idea   is to establish approximate controllability in small time to the points of the affine space $u_0+\HH_N$ by combining   Proposition~\ref{P:2.4} and an induction argument in $N$.   Then the saturation property  will imply    approximate controllability in small time  to any point of $H^s$. Finally, controllability in any time $T$ is proved by
steering the system close to the target $u_1$ in small time, then forcing  it to remain close to $u_1$  for a sufficiently long   time.    
  The accurate proof is divided into four steps. 

\medskip

 {\it Step~1.~Controllability in small time to  $u_0+\HH_0$.}~Let us assume for the moment that $u_0\in H^{s+1}$.
 First we prove that   problem~\eqref{0.1},~\eqref{0.4} is approximately controllable to the set $u_0+\HH_0$ in small time. More precisely,  we show that, 
for any $\e>0$, $\eta\in\HH_0$, and $T_0>0$, there is a time~$T<T_0$  and a   control $\hat \eta
\in \Theta(u_0,h,T)\cap L^2(J_T,\HH)$ such that 
$$
	\|\RR_T(u_0,0,h+\hat\eta) - u_0-\eta \|_s<\e.
$$
Indeed, applying Proposition~\ref{P:2.4} for the couple $ (\eta,0)$, we see that
$$ 	
 \RR_{\de}(u_0,0,h+\de^{-1}\eta)\to  u_0+\eta \quad\text{in $H^s$ as $\de\to 0^+$},
$$
which gives the required result with $\hat\eta=\delta^{-1}\eta$ and $T=\delta$.

\medskip
{\it Step~2.~Controllability in small time to  $u_0+\HH_N$.}~We argue by induction.~Assume that the approximate controllability of problem~\eqref{0.1},~\eqref{0.4}   to the set~$u_0+\HH_{N-1}$ is already proved. Let $\eta_1\in \HH_N$  be    of the   form
\begin{equation}\label{2.2}
\eta_1=\eta -\sum_{m=1}^n \zeta_m^p	
\end{equation}
  for some   integer $n\ge1$ and vectors~$\eta,\zeta_1, \ldots,\zeta_n\in \HH_{N-1}$.
Applying Proposition~\ref{P:2.4} for the couple $ (0,\zeta_1)$, we see that
\begin{equation}\label{2.3}
 \RR_{\de}(u_0,\de^{-1/p}\zeta_1,h)\to  u_0-\zeta_1^p \quad\text{in $H^s$ as $\de\to 0$}.
\end{equation}
  By the uniqueness of the solution of the Cauchy problem,   the following  equality holds
$$
\RR_{t}(u_0+\delta^{-1/p}\zeta_1,0,h)=\RR_{t}(u_0,\delta^{-1/p}\zeta_1,h)+\delta^{-1/p}\zeta_1, \quad t\in J_\delta.
$$Taking $t=\delta$ in this equality and  using limit \eqref{2.3}, we obtain
$$
\|\RR_{\delta}(u_0+\delta^{-1/p}\zeta_1,0,h)-u_0+\zeta_1^p-\delta^{-1/p}\zeta_1\|_s\to 0\quad \text{as $\delta\to 0.$}	
$$
Combining this with the fact that $\eta,\zeta_1\in \HH_{N-1}$,   the   induction   hypothesis, and Proposition~\ref{P:1.1}, we find a small time $T>0$ and  a   control $ \hat \eta_1
\in \Theta(u_0,h,T)\cap L^2(J_T,\HH)$ such that 
$$
	\|\RR_T(u_0,0,h+\hat\eta_1) - u_0-\eta+\zeta_1^p \|_s<\e.
$$
Iterating this argument successively for the vectors  $\zeta_2, \ldots, \zeta_n$, we construct a small time 
$\hat T >0$ and  a   control $\hat \eta
\in \Theta(u_0,h, \hat T)\cap L^2(J_{\hat T},\HH)$ satisfying 
$$
	\|\RR_{\hat T}(u_0,0,h+\hat\eta) -  u_0-\eta+\zeta_1^p+\ldots+\zeta_n^p \|_s=\|\RR_{\hat T}(u_0,0,h+\hat\eta) -  u_0-\eta_1  \|_s<\e,
$$where we used \eqref{2.2}. 
This proves the approximate controllability in small time to any point in  $u_0+\HH_N$.

\medskip
{\it Step~3. Global   approximate  controllability in small time.} Now 
let $u_1\in H^s$ be arbitrary. As~$\HH_\ty$ is dense in $H^s$, there is an integer $N\ge1$ and point $\hat u_1\in u_0+\HH_N$ such that 
\begin{equation}\label{2.4}
\|u_1-\hat u_1\|_s<\e/2.
\end{equation}
By the results of Steps 1 and 2, for any $\e>0$ and $T_0>0$ there is a time  
   $T <T_0$  and a   control $\hat \eta
\in \Theta(u_0,h,T)\cap L^2(J_T,\HH)$ satisfying 
$$	\|\RR_T(u_0,0,h+\hat\eta) - u_1 \|_s<\e/2.
$$  Combining this with \eqref{2.4}, we get approximate controllability in small time to~$u_1$ from $u_0\in H^{s+1}$.  
Taking control equal to zero on a small time interval and using the  regularising property of the equation, we   conclude  small time approximate controllability   starting from arbitrary  $u_0\in H^s$. By regularising property we mean that the solution becomes smooth at any time $t>0$ when the initial point~$u_0$ is in $ H^s$. This can be seen, for example, by using the Duhamel formula and the regularising property of the   heat semigroup.

\medskip
{\it Step~4. Global   approximate  controllability in fixed time $T$.}
Since we have global controllability in small time, to complete the proof of the theorem, 
 it suffices to show that,  for any $\e,T>0$   and   any      $u_1\in H^s$,    there is  a   control $\eta
\in \Theta(u_1,h,T)\cap L^2(J_T,\HH)  $ such~that
\begin{equation}\label{ATT}
\|\RR_T(u_1,0,h+\eta) - u_1 \|_s<\e.
\end{equation}  Note that here  the initial condition and the target coincide   with $u_1$.   It is  not clear, whether  it is possible or not   to  find a control taking values in $\HH$ such that the solution starting from~$u_1$ remains close to that point on all the time interval~$J_T$. However, the argument below allows to have \eqref{ATT} precisely at time~$T$. 

 By Proposition~\ref{P:1.1}, there are numbers  $r\in (0,\e)$ and $\tau>0$   such that,   for any~$v \in B_{H^s}(u_1,r),$ we have  $(0,0)\in  \hat\Theta(v,h,\tau)$ and
$$
	\|\RR_t(v,0,h)-u_1\|_s<\e \quad \text{for }   \,\, t\in J_\tau.  
$$  Thus starting from any point $v$ in the $r$-neighborhood of $u_1$, we are guaranteed to remain $\e$-close to $u_1$  on the (uniform) time interval $[0,\tau].$ 
If $\tau>T$, then   the proof is complete. Otherwise,  applying the result of Step 3 with initial condition $u_0=\RR_\tau(v,h)$, small time $T'< T-\tau$, and target $u_1$, we find a   control $\hat \eta\in \Theta(u_0,h,T')\cap L^2(J_{T'},\HH)$ such that 
$$	\|\RR_{T'}(u_0,0,h+\hat\eta) - u_1 \|_s<r.
$$   Applying again Proposition~\ref{P:1.1}, we conclude that, if $2\tau+T'>T$, then   the proof is complete.  Otherwise, we apply  again  the small time controllability property to return to the ball $B_{H^s}(u_1,r)$. After a finite number (less than the integer part of   $T/\tau+1$) of iterations,  we   complete the proof of the theorem.
   \end{proof}
   Now let us assume that 
     the nonlinear term $f$ in Eq.~\eqref{0.1} is a polynomial of degree $p\ge2$, i.e., $g=0$.  
     Recall that the space $\HH(\II)$ is  defined by~\eqref{0.5} for  a    finite symmetric set~$\II\subset \Z^d$       containing the origin. Let us denote by~$\tilde \II$ the set   of   all 
      linear combinations of elements of~$\II$ with integer coefficients. By~definition,~$\II$ is a generator if $\tilde \II=\Z^d$. Let $H^s(\II)$ be the closure in $H^s$ of the~set
    $$
    \lspan\{\sin\lag x,k\rag, \cos\lag x,k\rag:  k\in \tilde\II\}.
    $$ 
   \begin{proposition}\label{P:4.1} 
The space    $\HH(\II)$ is saturating  if and only if~$\II$ is a generator.   
\end{proposition}
This proposition is established in Section~\ref{S:4}. 
     We have the following more detailed version of Theorem~\ref{T.2.3}. 
\begin{theorem}\label{T:2.5}
Let $h\in L^2  (J_T,H^{s-1}(\II))$ and let the function $f$  be a polynomial of degree~$p\ge2$. Then Eq.~\eqref{0.1} is approximately controllable if and only if  $\II$ is a generator.
\end{theorem}
    \begin{proof}  By Proposition~\ref{P:4.1},   
    if $\II$ is a generator, then $\HH(\II)$ is saturating. Thus Eq.~\eqref{0.1} is approximately controllable by Theorem~\ref{T.2.3}. 
    
    Now assume that $\II$   is not a generator. Then   there is a vector $m\in \Z^d$   which does not belong to   $\tilde \II$.   The   set of attainability from the origin defined by
    $$
   \aA:= \{\RR_T(0,0,h+\eta): \eta\in  \Theta(u_0,h,T)\cap L^2(J_T,\HH) \}  
    $$ is contained in  $H^s(\II)$. Indeed, this follows from the assumption that $h\in L^2  (J_T,H^{s-1}(\II))$, the fact that the space $H^s(\II)$ is invariant for the linear dynamics of the heat equation, and the assumption that the nonlinear term $f$ in Eq.~\eqref{0.1} is a polynomial (so $f$   maps  $H^s(\II)$ to itself).       Thus the functions $\cos\lag x,m\rag$ and $\sin\lag x,m\rag$ will be  orthogonal to $\aA$, hence $\aA$ is not dense in $H^s$.  This proves that    Eq.~\eqref{0.1} is not approximately controllable by~$\HH(\II)$-valued~control. 
 \end{proof}

 \section{Proof of Proposition~\ref{P:2.4}}\label{S:3}

  Assume that $u_0,\eta\in H^{s+1}$, $\zeta \in  H^{s+2}$, and   $h\in L^2  (J_T,H^{s-1})$ are such that 
\begin{equation}\label{SSSD}
\|u_0\|_{s+1} +\|\eta\|_{s+1}+\|\zeta\|_{s+2}+\|h\|_{L^2  (J_T,H^{s-1})}<R.
\end{equation}  Let us take any $\delta>0$ and
   consider the equation
\begin{equation}
	\p_t u-\nu\Delta (u+\de^{-1/p}\zeta)+f(u+\de^{-1/p}\zeta)=h+\de^{-1}\eta.
\label{3.1}
\end{equation}
 By Proposition~\ref{P:1.1}, problem~\eqref{3.1}, \eqref{0.4}   has a unique maximal solution   defined on an interval $[0,T_*)$, where $T_*:=T_*(u_0,\de^{-1/p}\zeta, h+\de^{-1}\eta)>0$. Moreover,  
  \begin{equation}\label{3.2}
 \| u(t)\|_s\to +\infty \quad \text{as $t\to T_*^-$},
\end{equation}when $T_*<\ty$.
   We will show that
   \begin{itemize}
   \item[$(1)$]   there is a number $\delta_0>0$ such that for any $\delta<\de_0$, the solution of Eq.~\eqref{3.1} is defined for $t\in[0,T_*)$ with $T_*>\delta$;  
      	   \item[$(2)$]     the following   limit   holds
  \begin{equation}\label{3.4}
 	 u(\delta)\to  u_0+\eta-\zeta^p \quad\text{in $H^s$ as $\de\to 0^+$}, 
 \end{equation}uniformly with respect to $u_0,\eta,\zeta$ and $h$ satisfying \eqref{SSSD}.
   \end{itemize}
    Inspired by some ideas from~\cite{JK-1985, AS-2006}, we make a time substitution and consider the functions
 \begin{align}
 	w(t)&:=   u_0+t(\eta -\zeta^p),\label{TTRER}\\
 	v(t)&:= u(\delta t)-w(t),\nonumber
 \end{align}which are well defined for $t<\delta^{-1}T_*$. 
Then $v$  is a solution of problem
 \begin{align}
	\p_t v-\nu \delta\Delta (v+w+\de^{-1/p}\zeta)+ \delta f(v+w+\de^{-1/p}\zeta)- \zeta^p&=\delta h,\label{3.5}\\
v(0)&=0.\label{3.6}
\end{align}
Taking the scalar product in $L^2$ of Eq.~\eqref{3.5} with $(-\Delta)^s v+v$ and using the Cauchy--Schwarz   inequality,  we obtain
\begin{align}
 \p_t \| v \|^2_s+  \nu \delta\|\nabla  v \|^2_{s}   &\le   C_1\delta\left(  \|w\|_{s+1}+  \|h\|_{s-1}\right)  \|\nabla v\|_{s} \nonumber\\&\quad   +  C_1\left(  \delta^{1-1/p} \|\zeta\|_{s+2}+\|\delta f(v+w+\de^{-1/p}\zeta)- \zeta^p\|_s\right)  \|v\|_s. \label{3.7}
\end{align}
From the   Young inequality and \eqref{SSSD} we derive 
\begin{align}\label{3.8} 
	   C_1 \delta\left(  \|w\|_{s+1}+ \|h\|_{s-1}\right)  \|\nabla v\|_{s}&   \le  \frac{ \nu \delta}{2}\|\nabla v\|_{s}^2 +  C_2 \delta  \left(\|w\|_{s+1}^2+     \|h\|_{s-1}^2\right)\nonumber\\& \le\frac{ \nu \delta}{2}\|\nabla v\|_{s}^2 +  C_3\delta \left(1 +      \|h\|_{s-1}^2\right)
,  
\end{align} for   $t\le 1\wedge (\delta^{-1}T_*)$, where the constant $C_3$ (and almost all   the constants $C_i$ below) depends on $R$.
  The nonlinear term is estimated as in Section~\ref{S:1}.  We assume that $\delta\le 1$ and estimate the polynomial part as follows:
\begin{equation}
 \|\delta P_p(v+w+\de^{-1/p}\zeta)- \zeta^p\|_s\le  C_4\delta^{1/p} \left(\|v\|_s^p   +1\right) \label{3.9}
\end{equation} for $ t\le 1\wedge (\delta^{-1}T_*)$, where we used \eqref{1.10}, \eqref{SSSD}, and \eqref{TTRER}. For the term $g$, we use the inequality (cf.  \eqref{EEgna})
$$
\|g(a)\|_s\le C_5 \, \|a\|_s^s, \quad a\in H^s,
$$    and  the assumptions that $p>s$,   $\delta\le 1$, and \eqref{SSSD}   to prove that:  
\begin{align}
   \|\delta g(v+w+\de^{-1/p}\zeta)\|_s\le  C_6\delta^{1/p} \left(\|v\|_s^s   +1\right) \label{3.9SSS}
\end{align} for $ t\le 1\wedge (\delta^{-1}T_*)$,
Combining  inequalities \eqref{3.7}-\eqref{3.9SSS} and the assumption that~$p\ge2$,   we arrive at
\begin{align}
 \p_t \| v \|^2_s    \le     C_7 \delta^{1/p}\left(       \|h\|_{s-1}^2 + \|v\|_s^{p+1}+1   \right)  =C_7\delta^{1/p}\left(   \psi+ \|v\|_s^{p+1}    \right), \label{3.10}
\end{align}where  $ \psi:=   \|h\|_{s-1}^2 +   1$ and the constant     $C_7>0$ does not depend on     $\delta$. Let~us~set  
\begin{equation}
\Phi(t):= A+C_7 \delta^{1/p}\int_0^t  \|v\|_s^{p+1} \dd \tau,     \label{3.11}
\end{equation}
where 
$$
A:= C_7\delta^{1/p}\int_0^1 \psi \dd \tau.  	
$$   Inequality  \eqref{3.10} and initial condition  \eqref{3.6} imply  that 
$$
(\dot \Phi)^{2/(p+1)}\le (C_7 \delta^{1/p})^{2/(p+1)}\Phi. 
$$ 
The latter is  equivalent to 
$$
\frac{\dot \Phi}{\Phi^{(p+1)/2}}\le C_7\delta^{1/p}.
$$Integrating this inequality, we obtain
$$
\Phi(t)\le  A \left(1-\frac{(p-1)}{2}C_7 \delta^{1/p} A^{(p-1)/2} t\right)^{-2/(p-1)}, \quad t<1\wedge (\delta^{-1}T_*)\wedge T_1,
$$where 
\begin{align*}
T_1:=&\,    \left((p-1)C_7 \delta^{1/p} A^{(p-1)/2}\right)^{-1}\\=&  \left((p-1)C_7^{(p+1)/2} \delta^{(p+1)/(2p)} \left(\int_0^1\psi\dd \tau\right)^{(p-1)/2}\right)^{-1}.
\end{align*}  
  We   choose $\delta_0\in (0,1)$ so small that
  $T_1\ge1$ for any~$\delta<\delta_0$~and  
\begin{equation}\label{3.12}
\Phi(t)\le  2A= 2C_7\delta^{1/p}\int_0^1 \psi \dd \tau,  \quad t< 1 \wedge (\delta^{-1}T_*).
\end{equation} 
 From this and \eqref{3.2} we derive that 
$$
\delta^{-1}T_*>1 \quad\text{ for $\delta<\delta_0$},
$$which yields (1). Combining \eqref{3.10}-\eqref{3.12}, we see that 
$$
\|v(1)\|_s\le C_8 \delta^{1/p},
$$
so
$v(1)\to 0$ in $H^s$ as $\delta\to 0^+.$ This gives \eqref{3.4} and completes the proof of the proposition.

\section{Saturating subspaces}\label{S:4}
 
 Let $\HH(\II)$ be the space defined by \eqref{0.5} and $\II\subset \Z^d$ be   a   finite symmetric set      containing the origin. Here we prove Proposition~\ref{P:4.1}.
 \begin{proof}[Proof of Proposition~\ref{P:4.1}]
   {\it Step~1.~Sufficiency of the condition.}~Assume that $\II$ is a generator and~$\{\HH_k(\II)\}$ and $\HH_\ty(\II)$ are the vector  spaces defined by~\eqref{2.1} with~$\HH=\HH(\II)$. 
    We distinguish two cases.
   
   \medskip
 {\it Case~1.~$p$ is odd.}~This  case   is particularly simple   due to the following representation of the space $\FF(\HH)$.
\begin{lemma}\label{L:4.2}
 If $p$ is odd,   then
\begin{align}
\FF(\HH)&=\lspan\left\{\HH,\{\zeta^p:\,\,\zeta\in\HH \}\right\}\label{4.1}	\\
&=\lspan\left\{\HH,\{\zeta_1\cdot\ldots\cdot\zeta_p:\,\,\zeta_i\in\HH,\,\, i=1,\ldots,p \}\right\}.\label{4.2}
\end{align} 
	\end{lemma} This lemma is proved at the end of this section.  Here we apply it to show~that   
   \begin{equation}\label{4.3}
   	\cos\lag x,l\pm m\rag,\,\, \sin\lag x,l\pm m\rag \in \HH_1(\II) \quad \text{for $l,m\in \II$}.
   \end{equation} Indeed, this easily follows from \eqref{4.2} by taking $\zeta_1= \ldots= \zeta_{p-2}=1$ and choosing appropriately $\zeta_{p-1}$ and $\zeta_p$     from the identities
   \begin{align*}
   	\cos\lag x,l\pm m\rag&=\cos\lag x,l\rag\cos\lag x, m\rag\mp \sin\lag x,l\rag\sin\lag x, m\rag,\\
 	\sin\lag x,l\pm m\rag&=\sin\lag x,l\rag\cos\lag x, m\rag\pm \cos\lag x,l\rag\sin\lag x, m\rag.  
   \end{align*}Combining \eqref{4.3} with the fact that $\II$ is a generator, we see that 
   $$
   \HH_\ty(\II)=\lspan\{\cos\lag x,  m\rag, \, \sin\lag x,  m\rag:\,\, m\in \Z^d\}.
   $$Thus $ \HH_\ty(\II)$   is dense in $H^s$ and $\HH(\II)$ is saturating.
 	
   \medskip
    {\it Case~2.~$p$ is even.} 
 Here we show that
    \begin{equation}\label{4.4}
   	\cos\lag x,l\pm m\rag,\, \sin\lag x,l\pm m\rag \in \HH_2(\II) \quad \text{for $l,m\in \II$}.
   \end{equation} 
    We first  assume that $p\ge4$.  Let us   check that 
 \begin{equation}\label{4.5}
   	\cos\lag x,2 m\rag    \in \HH_1(\II) \quad \text{for $m\in \II$}.
   \end{equation} 
 Indeed, for any $\e>0$   and $\alpha=-2/(p-2)$,    we have the equalities
  \begin{align*}
   	(\e^\alpha\!+\!\e\cos\lag x,m\rag)^p&= \frac{p(p-1)}{4}\left(1\!+\!\cos\lag x, 2 m\rag\right) \!+\!\e^{\alpha p}\!+\!p\e^{\alpha (p-1) +1}\cos\lag x,m\rag\!+\!a(\e),\\
 		(\e^\alpha\!+\!\e\sin\lag x,m\rag)^p&= \frac{p(p-1)}{4}\left(1\!-\!\cos\lag x, 2 m\rag\right) \!+\!\e^{\alpha p}\!+\!p\e^{\alpha (p-1) +1}\sin\lag x,m\rag\!+\!b(\e),  
   \end{align*}where   $a(\e),\, b(\e)\to 0$ in~$H^s$ as $\e\to 0.$ As  
   $$
       1,\,  \cos\lag x, m\rag,\, \sin\lag x,m\rag  \in\HH(\II),
   $$
  we obtain \eqref{4.5}. 
  Now we use a similar  argument to prove~\eqref{4.4}. The fact that   
\begin{equation}\label{RRTE}
   	\cos\lag x,l+ m\rag \in  \HH_2(\II) \quad \text{for $l,m\in \II$}
\end{equation}is   checked by using the equalities 
    \begin{align*}
   	(\e^\alpha+\e(\cos\lag x,l\rag\pm\cos\lag x,m\rag))^p&=\pm p(p-1) \cos\lag x,  l\rag\cos\lag x,  m\rag+\eta^c_{\pm}(\e)  +a_\pm(\e),\\
 		(\e^\alpha+\e(\sin\lag x,l\rag\pm\sin\lag x,m\rag))^p&=\pm p(p-1) \sin\lag x,  l\rag\sin\lag x,  m\rag+\eta^s_{\pm}(\e)  +b_\pm(\e),  
   \end{align*}where $\eta^c_\pm(\e), \eta^s_\pm(\e)\in \HH_1(\II)$ (here we use \eqref{4.5}) and $a_\pm(\e), b_\pm(\e)\to 0$ in~$H^s$ as~$\e\to 0$. 
   The remaining assertions in  \eqref{4.4} are proved in a similar way.   From~\eqref{4.4} and the fact that $\II$ is a generator we derive    
 $$
   \HH_\ty(\II)\supset\lspan\{\cos\lag x,  l\rag,\, \sin\lag x,  l\rag:\,\, l\in \Z^d\},
   $$which proves that    $\HH(\II)$ is saturating  when $p\ge 4$.
   
     The case $p=2$ is easier. To show \eqref{4.5}, we use the equalities 
$$
   	\cos^2\lag x,m\rag  = \frac12(1+\cos\lag x,2m\rag) , 
 		 	\quad \sin^2\lag x,m\rag  = \frac12(1-\cos\lag x,2m\rag)   
$$and the assumption that $1\in \HH(\II)$. 
It follows that 
$$
\cos^2\lag x,m\rag,\, \sin^2\lag x,m\rag \in \HH_1(\II) \quad \text{ for $m\in \II$.}
$$ 
Then~\eqref{RRTE} follows from the equalities
  \begin{align*}
   	(\cos\lag x,l\rag\pm\cos\lag x,m\rag)^2&=\cos^2\lag x,l\rag +\cos^2\lag x,m\rag\pm2\cos\lag x,l\rag \cos\lag x,m\rag,\\
 		(\sin\lag x,l\rag\pm\sin\lag x,m\rag)^2&=\sin^2\lag x,l\rag +\sin^2\lag x,m\rag\pm2\sin\lag x,l\rag \sin\lag x,m\rag.  
   \end{align*} The proof of the other  assertions in  \eqref{4.4} is similar. As in the case $p\ge4$, we conclude that  $\HH(\II)$ is saturating.

\medskip
 {\it Step~2.~Necessity of the condition.} Now assume that $\II$ is not a generator.  Then there is a vector~$m\in \Z^d$   which does not belong to the set $\tilde \II$ of all      linear~combinations of elements of~$\II$ with integer coefficients. It is easy to see that 
   $$
   \HH_\ty(\II)\subset\lspan\{\cos\lag x,  l\rag,\, \sin\lag x,  l\rag:\,\, l\in \tilde\II\}.
   $$Thus the functions $\cos\lag x,m\rag$ and $\sin\lag x,m\rag$ are orthogonal to $\HH_\ty(\II)$. This  implies that $\HH_\ty(\II)$ is not dense in~$H^s$, so $\HH (\II)$ is not saturating.     
\end{proof}

The simplest example of saturating space of form \eqref{0.5} will be the $(2d+1)$-dimensional space   corresponding to the set 
$$
\II=\{0, \pm e_i:\,\, i=1, \ldots, d\} \subset \Z^d, 
$$  where  $\{e_j\}_{j=1}^d$  is the standard basis in $\R^d$.   The following result, combined with Proposition~\ref{P:4.1}, gives a simple way  for constructing more saturating spaces.   
\begin{theorem}   A set  $\II\subset \Z^d$ is a generator if and
only if the greatest common divisor of the set 
$$
\{\det(a_1,\ldots,a_d): a_i\in \II,\,i=1, \ldots,d \}
$$ is $1$, where $\det(a_1,\ldots,a_d)$ is the determinant of the $d\times d$
matrix with columns $a_1,\ldots,a_d$.  
\end{theorem} 
  See        Section~3.7 in~\cite{Jac-85} for the proof of this theorem.

\begin{proof}[Proof of Lemma~\ref{L:4.2}]
Equality \eqref{4.1} follows immediately from the fact that~$p$ is odd.
Let us denote by~$\GG_1$ and $\GG_2$ the spaces on the right-hand sides of \eqref{4.1} and \eqref{4.2}, respectively. Obviously,~$\GG_1\subset \GG_2$. To see the inclusion  $\GG_2\subset \GG_1$, let us take any~$\zeta_1, \ldots, \zeta_p\in \HH$, and    consider the function
$$
F:\R^p\to \GG_1,\quad (x_1,\ldots,x_p)\mapsto (x_1\zeta_1+\ldots+x_p\zeta_p)^p.
$$
As $\GG_1$ is closed, it contains  the derivative 
$$
\frac{\p^p}{\p_{x_1}\ldots \p_{x_p}} F(0,\ldots,0)=p!\,\zeta_1\cdot\ldots\cdot \zeta_p.
$$  This implies that $\zeta_1\cdot\ldots\cdot \zeta_p\in \GG_1$, so 
    $\GG_2\subset \GG_1$.
 \end{proof}

\addcontentsline{toc}{section}{Bibliography}
\def\cprime{$'$} \def\cprime{$'$}
  \def\polhk#1{\setbox0=\hbox{#1}{\ooalign{\hidewidth
  \lower1.5ex\hbox{`}\hidewidth\crcr\unhbox0}}}
  \def\polhk#1{\setbox0=\hbox{#1}{\ooalign{\hidewidth
  \lower1.5ex\hbox{`}\hidewidth\crcr\unhbox0}}}
  \def\polhk#1{\setbox0=\hbox{#1}{\ooalign{\hidewidth
  \lower1.5ex\hbox{`}\hidewidth\crcr\unhbox0}}} \def\cprime{$'$}
  \def\polhk#1{\setbox0=\hbox{#1}{\ooalign{\hidewidth
  \lower1.5ex\hbox{`}\hidewidth\crcr\unhbox0}}} \def\cprime{$'$}
  \def\cprime{$'$} \def\cprime{$'$} \def\cprime{$'$}

\end{document}